\documentclass[12pt]{article}
\usepackage{etoolbox}
\def\row#1/#2!{#1_{\IfStrEq{#2}{}{n}{#2}} & \dynkin{#1}{#2}\\}

\usepackage{}
\usepackage[numbers,sort&compress]{natbib}

\usepackage{color}
\usepackage{threeparttable}

\usepackage[english]{babel}
\usepackage{t1enc}
\usepackage{amsthm,amsmath,amssymb}
\usepackage{mathrsfs}
\usepackage[mathscr]{eucal}
\usepackage[latin1]{inputenc}
\usepackage[english]{babel}
\usepackage{latexsym}
\usepackage{graphicx}
\usepackage[noend]{algpseudocode}
\usepackage{algorithmicx,algorithm}
\usepackage{lineno}
\usepackage{booktabs}
\usepackage{multirow}
\newtheorem{theorem}{Theorem}
\newtheorem{corollary}{Corollary}
\newtheorem{proposition}{Proposition}
\newtheorem{lemma}{Lemma}
\newtheorem{problem}{Problem}
\newtheorem{remark}{Remark}

\newtheorem{definition}{Definition}
\def \T{\textup{T}}

\def \rank{\textup{rank}}

\def \diag{\textup{diag~}}
\def \Im{\textup{Im}}
\def \Ker{\textup{Ker}}

\def \span{\textup{span}}
\makeatletter
\newcommand{\rmnum}[1]{\romannumeral #1}
\newcommand\restr[2]{{
		\left.\kern-\nulldelimiterspace 
		#1 
		\right|_{#2} 
}}
\newcommand{\Rmnum}[1]{\expandafter\@slowromancap\romannumeral #1@}
 
\makeatother
\title{The Smith normal form of the walk matrix of the Dynkin graph $D_n$ for $n\equiv 0\pmod{4}$}
\author{\small Wei Wang\thanks{Corresponding author: wangwei.math@gmail.com}
\\
{\footnotesize$^{\rm a}$School of Mathematics, Physics and Finance, Anhui Polytechnic University, Wuhu 241000, P. R. China}
}
\date{}
\begin{document}
	\maketitle
	
	\begin{abstract}
		Let $W(D_n)$ denote the walk matrix of the Dynkin graph $D_n$. We prove that the Smith normal form of $W(D_n)$ is $$\diag[\underbrace{1,1,\ldots,1}_{\frac{n}{2}-1},\underbrace{2,2,\ldots,2}_{\frac{n}{2}-1},0,0]$$  when $n\equiv 0\pmod{4}$. This gives an affirmative answer to a question in  [W.~Wang, C.~Wang, S.~Guo, On the walk matrix of  the Dynkin graph $D_n$, Linear Algebra Appl.  653 (2022) 193--206].\\
		
		
	\end{abstract}
	\section{Introduction}
	For an $n$-vertex graph with adjacency matrix $A$, the \emph{walk matrix} of $G$ is
	\begin{equation}
	W:=[e,Ae,\ldots,A^{n-1}e],
	\end{equation}
	where $e$ is the all-one vector. This particular kind of matrices has received increasing attention in recent studies as it is closely related to many interesting properties of graphs, such as the controllability of graphs \cite{godsil2010} and the  spectral determination of graphs \cite{wang2006PHD,wang2014Ejc,wang2017JCTB,qiu2023DM,wang2021Eujc,wang2023Eujc}.
	
	For any $n\times n$ integral matrix $M$, there exist two unimodular matrix $P$ and $Q$ such that $PMQ$ has a diagonal form $\diag[d_1,d_2,\ldots,d_n]$ with  $d_i\mid d_{i+1}$ for $i=1,\ldots,n-1$. We always assume that each $d_i$ is nonnegative and make the usual convention that $0\mid 0$. The diagonal matrix $\diag[d_1,d_2,\ldots,d_n]$ is called the \emph{Smith normal form} of $M$; $d_i$ is called the $i$-th \emph{invariant factor} of $M$. These invariant factors are unique, in fact, $d_i=\Delta_i/\Delta_{i-1}$, where $\Delta_i$, called the $i$-th \emph{determinantal factor}, is the greatest common divisor of all $i\times i$ minors of $M$ and $\Delta_0=1$ by convention. We note that the Smith normal form of $M$ is a refinement of $\det M$ and $\rank~M$, the determinant and  the rank of $M$. Precisely, $\det M=\pm d_1 d_2\cdots d_n$ and $\rank\,M=\max\{i\colon\,d_i\neq 0\}$. Moreover, for any prime $p$, $\rank_p M=\max\{i\colon\, d_i\not\equiv 0\pmod{p}\}$, where $\rank_p M$  denotes the rank of $M$ over the finite field $\mathbb{F}_p$.	
	
It is a classic problem to study the Smith normal form of various kinds of integral matrices associated with graphs. Although there have been many results on Smith normal forms of adjacency matrices and Laplacian matrices (see e.g. \cite{chandler2017DCC,lorenzini2008JCTB,hou2011ACS,pantangi2019JCTA}), the corresponding problem for walk matrices is still on an early stage due to its difficulty. For any graph $G$ with $n$ vertices,  Wang \cite{wang2014Ejc} showed that $\rank_2 W(G)\le  \lceil\frac{n}{2}\rceil$, i.e., at most $\lceil\frac{n}{2}\rceil$ invariant factors of $W(G)$ are odd.  Furthermore, Choi et.~al \cite{choi2021LAA} showed that, for any positive integer $k$, at most $\lfloor\frac{n}{2}\rfloor$ invariant factors of $W(G)$ are congruent to $2^k$ modulo $2^{k+1}$; this was conjectured in \cite{wang2021LAA} where the special case $k=1$ was proved. In a more recent paper, Wang et.~al \cite{wang2022LAA} determined the Smith normal form of the walk matrix of the Dynkin graph $D_n$ under the assumption that $4\nmid n$.

\begin{figure}
	\centering
	\includegraphics[height=2cm]{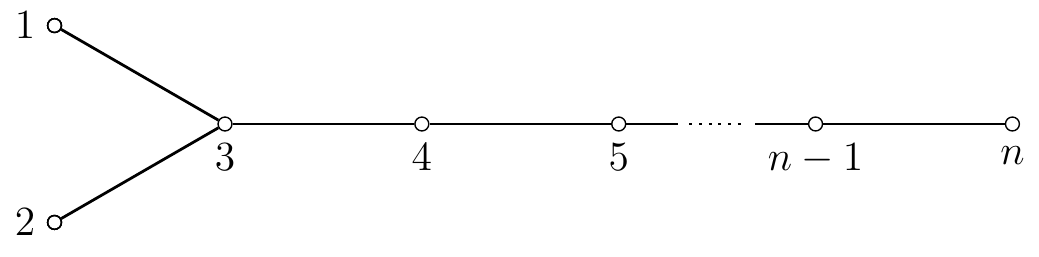}
	\caption{Graph ${D}_n(n\ge 4)$.}
\end{figure}
\begin{theorem}[\cite{wang2022LAA}]\label{snfd}
		If $4\nmid n$ then $W(D_n)$  has the Smith normal form $$\diag[\underbrace{1,1,\ldots,1}_{\lceil\frac{n}{2}\rceil},\underbrace{2,2,\ldots,2}_{\lfloor\frac{n}{2}\rfloor-1},0].$$
\end{theorem}
For the case that $n\equiv 0\pmod{4}$, only the rank of $W(D_n)$ is determined.
\begin{proposition}[\cite{wang2022LAA}]\label{rkW}
If $4\mid n$ then $\rank~W(D_n)=n-2$, or equivalently, exactly the last two invariant factors are zero.
\end{proposition}	
At the end of \cite{wang2022LAA}, the authors proposed the following
	\begin{problem}[\cite{wang2022LAA}]\normalfont{
			Is it true that  the Smith normal form of $W(D_n)$ is
			\begin{equation*} 
			\diag[\underbrace{1,1,\ldots,1}_{\frac{n}{2}-1},\underbrace{2,2,\ldots,2}_{\frac{n}{2}-1},0,0]
			\end{equation*}
			whenever $4\mid n$?}
	\end{problem}
The main result of this paper is the following theorem, which gives an affirmative answer to Problem 1.
\begin{theorem}\label{main}
	If $4\mid n$ then the Smith normal form of $W(D_n)$ is
		\begin{equation*}
	\diag[\underbrace{1,1,\ldots,1}_{\frac{n}{2}-1},\underbrace{2,2,\ldots,2}_{\frac{n}{2}-1},0,0].
	\end{equation*}
	\end{theorem}
The basic clue to prove this theorem can be described by the following simple lemma.
\begin{lemma}\label{clue}
	Let $n\equiv 0\pmod{4}$. If the following two conditions hold
	simultaneously:\\
	\textup{(\rmnum{1})} $\rank_2 (W(D_n))\le \frac{n}{2}-1$, and\\
	\textup{(\rmnum{2})} $\Delta_{n-2}(W(D_n))$ divides $2^{\frac{n}{2}-1}$,\\
	then the Smith normal form of $W(D_n)$ has the desired form as described in Theorem \ref{main}.
\end{lemma}
\begin{proof}
	Let $r=\rank_2(W(D_n))$. By Proposition \ref{rkW} and noting the second condition of this lemma, we obtain that the Smith normal form of $W(D_n)$ has the following pattern:
		\begin{equation*}
	\diag[\underbrace{1,1,\ldots,1}_{r},\underbrace{2^{k_1},2^{k_2},\ldots,2^{k_{n-r-2}}}_{n-r-2},0,0],
	\end{equation*}
	where $k_{n-r-2}\ge k_{n-r-3}\ge\cdots\ge k_1\ge 1$ and $2^{k_1+k_2+\cdots+k_{n-r-2}}=\Delta_{n-2}(W(D_n))$. As $r\le \frac{n}{2}-1$, we see that
	$k_1+k_2+\cdots+k_{n-r-2}\ge n-r-2\ge\frac{n}{2}-1$. But since $\Delta_{n-2}(W(D_n))$ divides $2^{\frac{n}{2}-1}$, we also have that $k_1+k_2+\cdots+k_{n-r-2}\le \frac{n}{2}-1$. Thus equalities must hold simultaneously in all these inequalities.  This proves that $W(D_n)$ has the desired Smith normal form.
\end{proof}
	\section{An upper bound on $\rank_2W(D_n)$}
	The main aim of this section is to show that the first condition stated in Lemma \ref{clue} always holds for $n\equiv 0\pmod{4}$.
	\begin{theorem}\label{rk2}
		$\rank_2 W(D_n)\le \frac{n}{2}-1$ for $n\equiv 0\pmod{4}$.
	\end{theorem}
\begin{lemma}\label{droot}
Let $\phi_n(x)=\det (xI-A(D_n))$ be the characteristic polynomial of $D_n$. If  $n$ is even then $x=0$ is a double root of $\phi_n(x)$ over $\mathbb{F}_2$.
\end{lemma}
\begin{proof}	
	For any positive integer $n$, let $\rho_n(x)$ be the characteristic polynomial of the $n$-vertex path graph $P_n$. Expanding $\det (xI-A(D_n))$ by the first row, one easily finds that
	\begin{equation}\label{dnpn}
\phi_n(x)=x\rho_{n-1}(x)-x\rho_{n-3}(x).
	\end{equation}
	It is well known that $\rho_n(x)=U_n(x/2)$, where $U_n(x)$ is the $n$-th Chebyshev polynomial of the second kind. Let $a_n$ and $b_n$ be the constant term and the coefficient of $x$ in $\rho_n(x)$, respectively.  By the recurrence relation $\rho_n(x)=x\rho_{n-1}(x)-\rho_{n-2}(x)$, we have
	\begin{equation}
	\begin{cases}
	a_n=-a_{n-2}\\
	b_n=a_{n-1}-b_{n-2}
	\end{cases}
	\end{equation}
	and hence 	\begin{equation}\label{abmod2}
	\begin{cases}
	a_n\equiv a_{n-2}\pmod {2}\\
	b_n\equiv a_{n-1}+b_{n-2}\pmod{2}.
	\end{cases}
	\end{equation}

By Eq.~\eqref{dnpn} and the first congruence in  Eq.~\eqref{abmod2}, one finds that  both the constant term and the coefficient of  $x$ in $\phi_n(x)$ are zero, over $\mathbb{F}_2$. Let $c_n$ be the coefficient of $x^2$ in $\phi_n(x)$. Then $c_n=b_{n-1}-b_{n-3}$. Using Eq.~\eqref{abmod2}, we have that $c_n\equiv a_{n-2}\pmod{2}$ and hence the sequence $\{c_n\}_{n\ge 4}$ has period 2, over $\mathbb{F}_2$. Note that $a_2=-1$. This means $c_n\equiv 1\pmod{2}$ whenever $n$ is even and hence the lemma follows.
\end{proof}
For an $n\times n$ matrix $M$ over a field $\mathbb{F}$, we identify $M$ with the linear transformation $\sigma:\mathbb{F}^n\mapsto \mathbb{F}^n$ defined by $\sigma(x)=Mx$ for all $x\in \mathbb{F}^n$. Let $\Im (M)$ and $\Ker (M)$ denote the image and kernel of  $M$, respectively. In symbols, $\Im (M)=\{Mx\colon\,x\in \mathbb{F}^n\}$ and $\Ker (M)=\{x\in\mathbb{F}^n\colon\,Mx=0\}$.
\begin{lemma}\label{disj}
If $n\equiv 0\pmod{4}$ then	$\Im(W(D_n))\cap \Ker (A(D_n))=\{0\}$ over $\mathbb{F}_2$.
\end{lemma}
\begin{proof}
	Let $n\equiv 0\pmod{4}$ and $\phi_n(x)$ be the characteristic polynomial of $D_n$. In the following argument we always assume the field $\mathbb{F}=\mathbb{F}_2$. By Lemma \ref{droot}, we can write
	\begin{equation}
	\phi_n(x)=x^2\psi(x),
	\end{equation}
with $\psi(0)\neq 0$. Thus, $x^2$ and $\psi(x)$ are coprime polynomials. Let  $A=A(D_n)$. By the primary decomposition theorem, we have
\begin{equation}\label{pdt}
\mathbb{F}_2^n=\Ker (A^2)\oplus \Ker( \psi(A)).
\end{equation}
Moreover, $\dim \Ker (A^2)=2$,  $\dim \Ker( \psi(A))=\deg \psi(x)=n-2$ and both subspaces are $A$-invariant.

\noindent\emph{Claim} 1. $\Ker(A^2)=\Ker(A)=\span\{\alpha,\beta\}$,
where $\alpha=(1,1,0,\ldots,0)^\T$ and $\beta=(0,1,\ldots,0,1)^\T.$

Direct calculation shows that $A\alpha=0$ and $A\beta=0$. Thus,  $\Ker(A^2)\supset\Ker(A)\supset\span\{\alpha,\beta\}$. Note that $\alpha$ and $\beta$ are linearly independent. Thus $\span\{\alpha,\beta\}$ is two-dimensional. Since $\Ker (A^2)$ is also two-dimensional, all the three spaces must be equal, as claimed.

\noindent\emph{Claim} 2. $\Ker (A^2)\perp \Ker( \psi(A))$, that is, $\xi^\T\eta=0$ for any $\xi\in \Ker (A^2)$ and $\eta\in \Ker( \psi(A))$.

As $x^2$ and $\psi(x)$ are coprime, there exist two polynomials $\mu(x)$ and $\nu(x)$ (in $\mathbb{F}_2[x]$) such that
$\mu(x)x^2+\nu(x)\psi(x)=1$.Noting that $A^\T=A$ and $A^2\xi=\psi(A)\eta=0$, we obtain
\begin{eqnarray*}
	\xi^\T\eta&=&	\xi^\T(\mu(A)A^2+\nu(A)\psi(A))\eta\\
	&=&(\mu(A)A^2\xi)^\T\eta+	\xi^\T\nu(A)\psi(A)\eta\\
	&=&0.
\end{eqnarray*}
\noindent\emph{Claim} 3. $e\in \Ker( \psi(A))$.

Let $U=\{z\in\mathbb{F}_2^n\colon\,\alpha^\T z=\beta^\T z=0\}$. As $\alpha^\T e=0$,  and $\beta^\T e=0$ (noting $4\mid n$), we see that $e\in U$. Clearly, $\dim U=n-2$ as $\alpha$ and $\beta$ are linearly independent. By Claims 1 and 2, we have $\span\{\alpha,\beta\}\perp\Ker (\psi(A))$ and hence $\Ker (\psi(A))\subset U$. We conclude that $\Ker (\psi(A))= U$ since they have the same dimension. Claim 3 follows.

As $\Ker (\psi(A))$ is $A$-invariant, Claim 3 implies $A^k e\in \Ker (\psi(A))$ for any nonnegative integer $k$. This means that all columns of $W(D_n)$ belong to  $\Ker (\psi(A))$, i.e., $\Im(W(D_n)) \subset \Ker (\psi(A))$. Lemma \ref{disj} follows by Eq.~\eqref{pdt}.
\end{proof}
We also need two simple facts noticed in \cite{wang2017JCTB}.
\begin{lemma}[\cite{wang2017JCTB}]\label{odd0}	 Let $G$ be an $n$-vertex graph and $\phi(x)=x^n+c_1x^{n-1}+\cdots+c_{n-1}x+c_{n}$ be the characteristic polynomial of $G$. Then $c_i\equiv 0\pmod{2}$ for any odd $i$ in $\{1,2,\ldots,n\}$.
\end{lemma}
\begin{lemma}[\cite{wang2017JCTB}]\label{me}
	Let $M$ be an integral symmetric matrix. If $M^2\equiv 0\pmod{2}$ then $Me\equiv 0\pmod{2}$.
\end{lemma}

\noindent\emph{Proof of Theorem \ref{rk2}}
Let $\phi(x)=x^n+c_1x^{n-1}+\cdots+c_{n-1}x+c_{n}$ be the characteristic polynomial of $D_n$.
 By Lemma \ref{odd0}, we find that $\phi(x)\equiv x^n+c_2x^{n-2}+\cdots+c_{n-2}x^2+c_n\pmod{2}$. Moreover, by Lemma \ref{droot}, we have $c_n\equiv 0\pmod{2}$. Therefore, we obtain
 \begin{eqnarray*}
 \phi(x)&\equiv&x^n+c_2x^{n-2}+\cdots+c_{n-2}x^2\\
 	&\equiv&(x^\frac{n}{2}+c_{2}x^{\frac{n}{2}-1}+\cdots+c_{n-2}x)^2\pmod{2}.
 \end{eqnarray*} Let $A=A(D_n)$ and
 $$\varphi(x)=x^\frac{n}{2}+c_{2}x^{\frac{n}{2}-1}+\cdots+c_{n-2}x.$$
Then $\phi(A)\equiv \varphi^2(A)\pmod{2}$. Noting that $\phi(A)=0$ by Cayley-Hamilton Theorem, it follows from Lemma \ref{me} that $\varphi(A)e\equiv 0\pmod{2}$.

 Let $$\varphi_1(x)=\frac{\varphi(x)}{x}=x^{\frac{x}{2}-1}+c_2x^{\frac{n}{2}-2}+\cdots+c_{n-2}.$$
 Then we have $A\varphi_1(A)e=\varphi(A)e\equiv 0\pmod{2}$, that is, $\varphi_1(A)e\in \Ker(A)$, over $\mathbb{F}_2$. Noting that $\varphi_1(A)e\in \Im(W(D_n))$, it follows from Lemma \ref{disj} that $\varphi_1(A)e=0$ over $\mathbb{F}_2$. This means that $A^{\frac{n}{2}-1}e$ (and hence $A^{k}e$ with any $k\ge \frac{n}{2}-1$) can be expressed as a linear combination of the $\frac{n}{2}-1$ vectors $e,Ae,\ldots,A^{\frac{n}{2}-2}e$. Thus, $\rank_2 (W(D_n))\le \frac{n}{2}-1$, as desired.

	\section{A lower bound on $\Delta_{n-2}(W(D_n))$}
We shall show that the second condition required in Lemma \ref{clue} is always satisfied.
	\begin{theorem}\label{up}
		$\Delta_{n-2}(W(D_n))\mid 2^{\frac{n}{2}-1}$ for $n\equiv 0\pmod{4}$.
	\end{theorem}
Let $n$ be a fixed positive integer satisfying $n\equiv 0\pmod{4}$ and we assume further that $n\ge 8$. (Theorem \ref{up} can be easily checked by hand for $n=4$.) Define two matrices

		$$B=\begin{bmatrix}
	0 &1 &{}&{}&{} &{}&{}  &{}      \\
	2 &0 &1 &{}&{} &{}&{}  &{}  \\
	{}&1 &0 & 1&{} &{}&{}  &{}\\
	{}&{}&1 & 0&1  &{}&{}  &{}  \\
{} &{}&{}  & \ddots       &\ddots    &\ddots &{}  \\
{} &{}&{}  &{} &1       &0     &1   \\
1  &0 &-1  &0           &\cdots &2     &0\\
	\end{bmatrix}_{(n-2)\times (n-2)}$$
	and
	$$C=\begin{bmatrix}
1&{}&{}&{}&{}&{}&{}\\
1&{}&{}&{}&{}&{}&{}\\
{}&1&{}&{}&{}&{}&{}\\
{}&{}&1&{}&{}&{}&{}\\
{}&{}&{}&1&{}&{}&{}\\
{}&{}&{}&{}&\ddots&{}&{}\\
{}&{}&{}&{}&{}&1&{}\\
{}&{}&{}&{}&{}&{}&1\\
1&0&-1&0&\cdots&1&0
	\end{bmatrix}_{n\times(n-2)},
	$$
	where the last rows (except the last two entries) of $B$ and $C$ repeat the same pattern $(1,0,-1,0)$ cyclically. The following identity can be verified in a routine way.	
	\begin{lemma}
	Let $A=A(D_n)$.  Then $AC=CB$.
	\end{lemma}
For an $m\times m$ matrix $M$, the walk matrix of $M$, denoted by $W(M)$, is the matrix $[e,Me,\ldots,M^{m-1}e]$. Let $\tilde{W}(D_n)$ denote the $(n-2)\times(n-2)$ matrix obtained from $[e,Ae,\ldots,A^{n-3}e]$ by removing the first and the last rows.
\begin{corollary}\label{wdb}
 $\tilde{W}(D_n)=W(B)$.
\end{corollary}
\begin{proof}
	As $AC=CB$, we have $A^kC=CB^k$ for any nonnegative integer $k$. Clearly, $Ce_{n-2}=e_{n}$. Thus, $A^ke_n=A^kCe_{n-2}=CB^ke_{n-2}$ and hence
	$$[e_n,Ae_n,\ldots,A^{n-3}e_n]=CW(B).$$
	Noting that removing both the first and the last rows from $[e_n,Ae_n,\ldots,A^{n-3}e_n]$  and  $C$ results in $\tilde{W}(D_n)$ and the identity matrix, we find that $\tilde{W}(D_n)=W(B)$, as desired.
\end{proof}
The following lemma is a slight modification of a result due to Mao, Liu and Wang \cite[Lemma 3.3]{mao2015LAA}. The current form appeared in \cite{wang2022LAA}.
	\begin{lemma}[\cite{mao2015LAA,wang2022LAA}]\label{relWA}
		Let $M$ (and hence $M^\T$) be a real $m\times m$  matrix which is diagonalizable over the real field $\mathbb{R}$. Let  $\xi_1,\xi_2,\ldots,\xi_m$ be  $m$ independent eigenvectors of $M^\T$ corresponding to  eigenvalues $\lambda_1,\lambda_2,\ldots,\lambda_m$, respectively. Then we have
		$$\det W(M)=\frac{\prod_{1\le k<j\le m} (\lambda_j-\lambda_k)\prod_{j=1}^m e^\T\xi_j}{\det[\xi_1,\xi_2,\ldots,\xi_m]}.$$
	\end{lemma}
\begin{definition}\normalfont{
	 For $k\in\{1,2,\ldots,n-1\}\setminus\{\frac{n}{2}\}$,	let $\alpha_k=\frac{2k-1}{2(n-1)}\pi$ and
		\begin{equation}\label{defxi}
		\xi_k=\begin{bmatrix}
		2\cos(n-3)\alpha_k\\
		2\cos(n-4)\alpha_k\\
		\vdots\\
		\vdots\\
		\vdots\\
		2\cos 2\alpha_k\\
		2\cos \alpha_k\\
		1		
		\end{bmatrix}+
		\begin{bmatrix}
			2\cos(n-7)\alpha_k\\
				\vdots\\
		2\cos\alpha_k\\
	
		1\\
		0\\
		0\\
		0\\
		0		
		\end{bmatrix}+\cdots+\begin{bmatrix}
				2\cos\alpha_k\\
				1\\
				0\\
				\vdots\\
		0\\
		0\\
		0\\
		0
		\end{bmatrix}.
		\quad (\text{sum of $\frac{n}{4}$ terms})
		\end{equation}
}
\end{definition}
\begin{remark}\label{usm}
Let $U$ be the  upper shift matrix: 
\begin{equation*}\begin{small}
U=\begin{bmatrix}
0&1&&&\\
&0&1&&\\
&&\ddots&\ddots&\\
&&&0&1\\
&&&&0
\end{bmatrix}_{(n-2)\times (n-2)}.
\end{small}	\end{equation*} 
 Then we can rewrite Eq.~\eqref{defxi} in a more compact way :
$$\xi_k=\sum\limits_{i=0}^{\frac{n}{4}-1}U^{4i}\tau_k,$$
where $\tau_k$ is the first term of the summation in Eq.~\eqref{defxi}.
\end{remark} 
\begin{lemma}
 $B^\T \xi_k=(2\cos \alpha_k)\xi_k$ 	for each $k\in\{1,2,\ldots,n-1\}\setminus\{\frac{n}{2}\}$.
\end{lemma}
\begin{proof}
	Let $m=n/4$ and $\eta=[c_m,b_m,a_m,d_{m-1},c_{m-1},b_{m-1},a_{m-1},\ldots,d_1,c_1,b_1,a_1]^\T\in \mathbb{R}^{n-1}$, where
	\begin{equation}
	\begin{cases}
	a_i=1+\sum\limits_{t=1}^{i-1}2\cos 4t\alpha_k,\quad i=1,\ldots,m\\
	b_i=\sum\limits_{t=0}^{i-1}2\cos (4t+1)\alpha_k,\quad i=1,\ldots,m\\
	c_i=\sum\limits_{t=0}^{i-1}2\cos (4t+2)\alpha_k,\quad i=1,\ldots,m\\
	d_i=\sum\limits_{t=0}^{i-1}2\cos (4t+3)\alpha_k,\quad i=1,\ldots,m-1.\\
	\end{cases}
	\end{equation}
	Note that removing the first entry $c_m$ from $\eta$ results in  $\xi_k$.
	
\noindent\emph{Claim} 1. For each $i\in\{2,3,\ldots,m\}$, it holds that
	\begin{equation}\label{abcd}
\begin{cases}
c_i+a_i+a_1=(2\cos \alpha_k)b_i,\\
b_i+d_{i-1}=(2\cos \alpha_k)a_i,\\
a_i+c_{i-1}-a_1=(2\cos\alpha_k)d_{i-1},\quad \\
d_{i-1}+b_{i-1}=(2\cos\alpha_k)c_{i-1}.\\
\end{cases}
\end{equation}

We only verify the first equality and the other equalities can be settled in a similar way. Note that $a_1=1$. It follows from the formula $\cos x+\cos y=2\cos\frac{x+y}{2}\cos\frac{x-y}{2}$ that \begin{eqnarray*}
c_i+a_i+a_1&=&\sum\limits_{t=0}^{i-1}2\cos (4t+2)\alpha_k+2+\sum\limits_{t=1}^{i-1}2\cos 4t\alpha_k\\
	&=&2\sum\limits_{t=0}^{i-1}(\cos (4t+2)\alpha_k+\cos 4t\alpha_k)\\
	&=&2\cos\alpha_k \sum\limits_{t=0}^{i-1}2\cos(4t+1)\alpha_k\\
	&=&(2\cos\alpha_k) b_i,
\end{eqnarray*}
as desired.

\noindent\emph{Claim} 2. $c_m=a_m.$

By the formula $\cos x\sin y=\frac{1}{2}(\sin(x+y)-\sin(x-y))$, we have
$$c_m\cdot\sin2\alpha_k=\sum\limits_{t=0}^{m-1}2\cos (4t+2)\alpha_k\cdot\sin 2\alpha_k=\sin 4m\alpha_k,$$
and
$$a_m\cdot\sin2\alpha_k=\sin 2\alpha_k+\sum\limits_{t=1}^{m-1}2\cos 4t\alpha_k\cdot\sin 2\alpha_k=\sin (4m-2)\alpha_k.$$

Recall that $\alpha_k=\frac{2k-1}{2(n-1)}\pi$ and $k\in \{1,2,\ldots,n-1\}\setminus\{\frac{n}{2}\}$. As
$4m\alpha_k+(4m-2)\alpha_k=(2n-2)\alpha_k=(2k-1)\pi$, we obtain that $\sin 4m\alpha_k=\sin (4m-2)\alpha_k$. Moreover, it is easy to see that $2\alpha_k\in (0,\pi)\cup(\pi,2\pi)$ and hence $\sin 2\alpha_k\neq 0$. Thus, we must have $c_m=a_m$, as desired.

By Claim 2, the first equality in Eq.~\eqref{abcd} for $i=m$ reduces to $2a_m+a_1=(2\cos\alpha_k)b_m$. Consequently, we can write Eq. \eqref{abcd} in matrix form as follows:

\begin{equation}\label{Bt1}
	\begin{bmatrix}
	0 &2 &{}&{}&{} &{}&{}  &{1}      \\
	1 &0 &1 &{}&{} &{}&{}  &{0}  \\
	{}&1 &0 & 1&{} &{}&{}  &{-1}\\
	{}&{}&1 & 0&1  &{}&{}  &{0}  \\
	{} &{}&{}  & \ddots       &\ddots    &\ddots &{}&{\vdots}  \\
	{} &{}&{}  &{} &1       &0     &1  &0
	\end{bmatrix}_{(n-4)\times (n-2)}
	\begin{bmatrix}
	b_m\\a_{m}\\\vdots\\\vdots\\d_1\\c_1\\b_1\\a_1
	\end{bmatrix}
	=2\cos \alpha_k\begin{bmatrix}
	b_m\\a_{m}\\\vdots\\\vdots\\d_1\\c_1
	\end{bmatrix}.
\end{equation}
Note that
\begin{equation}\label{Bt2}
\begin{bmatrix}
{0} &{0}&\cdots   & \cdots       &0   &1 &0&2  \\
{0} &{0}&\cdots   &{\cdots} &0       &0     &1  &0
\end{bmatrix}_{(n-4)\times 2}
\begin{bmatrix}	b_m\\a_{m}\\\vdots\\\vdots\\d_1\\c_1\\b_1\\a_1
\end{bmatrix}
=\begin{bmatrix}
c_1+2a_1\\
b_1
\end{bmatrix}
=\begin{bmatrix}
2\cos 2\alpha_k+2\\
2\cos\alpha_k
\end{bmatrix}=
2\cos \alpha_k\begin{bmatrix}
b_1\\a_1
\end{bmatrix}.
\end{equation}
Combining Eqs.~\eqref{Bt1} and \eqref{Bt2} leads to $B^\T \xi_k=(2\cos\alpha_k)\xi_k$. The proof is complete.
\end{proof}
	\begin{lemma}[\cite{wang2022LAA}]\label{prodsin}For any $m\ge 2$,
	$$\prod_{j=1}^{m-1}\sin\frac{2j-1}{4(m-1)}\pi=2^{\frac{3}{2}-m}.$$
\end{lemma}

\begin{lemma}\label{etxi}
$	\prod\limits_{\substack{k=1\\k\neq\frac{n}{2}}}^{n-1}e^\T\xi_k=\pm 2^{\frac{n}{2}-1}.$
\end{lemma}
\begin{proof}
	By the formula $\cos x\sin y=\frac{1}{2}(\sin(x+y)-\sin(x-y))$, we find that, for any positive integer $m$,
\begin{equation}\label{sumcos}
(1+2\cos\alpha_k+2\cos 2\alpha_k+\cdots+2\cos q\alpha_k)\sin \frac{1}{2}\alpha_k=\sin(q+\frac{1}{2})\alpha_k.
\end{equation}	
Write $S_k^{(q)}=1+2\cos\alpha_k+2\cos 2\alpha_k+\cdots+2\cos q\alpha_k$. Noting that
$$e^\T\xi_k=S_k^{(n-3)}+S_k^{(n-7)}+\cdots+S_k^{(1)},$$
it follows from Eq.~\eqref{sumcos} that
\begin{equation}\label{exi}
e^\T \xi_k=\frac{\sin(n-\frac{5}{2})\alpha_k+\sin(n-\frac{13}{2})\alpha_k+\cdots+\sin \frac{3}{2}\alpha_k}{\sin\frac{1}{2}\alpha_k}.
\end{equation}
Using the formula $\sin x \sin y=-\frac{1}{2}(\cos(x+y)-\cos(x-y))$, we obtain
$$(\sin(n-\frac{5}{2})\alpha_k+\sin(n-\frac{13}{2})\alpha_k+\cdots+\sin \frac{3}{2}\alpha_k)\sin 2\alpha_k=-\frac{1}{2}(\cos(n-\frac{1}{2})\alpha_k-\cos\frac{1}{2}\alpha_k)$$
and Eq.~\eqref{exi} reduces to
\begin{equation}\label{exi2}
e^\T \xi_k=-\frac{\cos(n-\frac{1}{2})\alpha_k-\cos\frac{1}{2}\alpha_k}{2\sin\frac{1}{2}\alpha_k\sin 2\alpha_k}=\frac{\sin\frac{n-1}{2}\alpha_k\sin\frac{n}{2}\alpha_k}{\sin\frac{1}{2}\alpha_k\sin 2\alpha_k}.
\end{equation}

\noindent\emph{Claim}. It holds that \begin{equation}\label{qu1}
\prod_{\substack{k=1\\k\neq\frac{n}{2}}}^{n-1}\frac{\sin\frac{n}{2}\alpha_k}{\sin 2\alpha_k}=\pm 1.
\end{equation}

Recall that $\alpha_k=\frac{2k-1}{2(n-1)}\pi$. Thus  $2\alpha_k=\frac{2k-1}{n-1}\pi$ and $\frac{n}{2}\alpha_k=\frac{n}{4}\frac{2k-1}{n-1}\pi$.
Note that $\gcd(2,n-1)=\gcd(\frac{n}{4},n-1)=1$ (as $4\mid n$). Thus, if $I$ is a complete system of residues modulo $n-1$ then so are  $\{2k-1\colon\,k\in I\}$ and $\{\frac{n}{4}(2k-1)\colon\,k\in I\}$. Moreover, both equations $2k-1\equiv 0\pmod{n-1}$ and $\frac{n}{4}(2k-1)\equiv 0\pmod{n-1}$  have a unique solution $k=\frac{n}{2}$ in the range $\{1,2,\ldots,n-1\}$. It follows that
\begin{equation*}\prod_{\substack{k=1\\k\neq\frac{n}{2}}}^{n-1}\sin\frac{n}{2}\alpha_k \quad \text{ and}\quad
\prod_{\substack{k=1\\k\neq\frac{n}{2}}}^{n-1} \sin 2\alpha_k
\end{equation*}
equal
\begin{equation*}
\pm \prod_{k=1}^{n-2} \sin \frac{k}{n-1}\pi.
\end{equation*}
This proves the claim.

Finally, as $\frac{n-1}{2}\alpha_k=\frac{2k-1}{4}\pi$, we find that $\sin\frac{n-1}{2}\alpha_k=\pm\frac{1}{2}\sqrt{2}$ and hence
\begin{equation}\prod_{\substack{k=1\\k\neq\frac{n}{2}}}^{n-1}\sin\frac{n-1}{2}\alpha_k =\pm 2^{1-\frac{n}{2}}.
\end{equation} Moreover,
\begin{equation}\label{ps2}\prod_{\substack{k=1\\k\neq\frac{n}{2}}}^{n-1}\sin\frac{1}{2}\alpha_k =\prod_{\substack{k=1\\k\neq\frac{n}{2}}}^{n-1}\sin\frac{2k-1}{4(n-1)}\pi=\sqrt{2} \prod_{k=1}^{n-1}\sin\frac{2k-1}{4(n-1)}\pi=2^{2-n},
\end{equation}
where the last equality follows from Lemma \ref{prodsin}. It follows from Eqs.~\eqref{exi2}-\eqref{ps2} that
$$\prod\limits_{\substack{k=1\\k\neq\frac{n}{2}}}^{n-1}e^\T\xi_k=\pm 2^{\frac{n}{2}-1},$$
completing the proof.
\begin{lemma}\label{detcos}
	It holds that
	\begin{equation}\label{vdm}
	\begin{small}
	\begin{vmatrix}
	1&1&\cdots&1\\
	2\cos \theta_1&2\cos \theta_2&\cdots&2\cos \theta_m\\
		2\cos 2\theta_1&2\cos 2\theta_2&\cdots&2\cos 2\theta_m\\
	\vdots&\vdots&&\vdots\\
	2\cos (m-1)\theta_1&2\cos (m-1)\theta_2&\cdots&2\cos (m-1)\theta_m\\
	\end{vmatrix}=	\prod\limits_{1\le j< i\le m}(2\cos \theta_i-2\cos \theta_j).
\end{small}
\end{equation}
\end{lemma}
\begin{proof}
Let $k\in\{1,2,\ldots,m-1\}$. Note that $\cos k\theta=T_k(\cos \theta)$, where $T_k(x)$ is the $k$-th Chebyshev polynomial of the first kind \cite{rivlin}. As the leading term of $T_k(x)$ is  $2^{k-1}x^{k}$, we obtain
 $$\cos k\theta-2^{k-1}(\cos \theta)^k\in \span\{1,\cos\theta,\ldots,(\cos\theta)^{k-1}\},$$
 or equivalently,
 $$2\cos k\theta-(2\cos \theta)^k\in \span\{1,2\cos\theta,\ldots,(2\cos\theta)^{k-1}\}.$$
 This means that there exist some constants $c_{k,0},c_{k,1},\ldots,c_{k,k-1}$ such that 
 \begin{equation*}
 2\cos k\theta-c_{k,0}-c_{k,1}(2\cos\theta)-\cdots-c_{k,k-1}(2\cos\theta)^{k-1}=(2\cos\theta)^k.
 \end{equation*}
 Therefore, using some evident row operations iteratively, the determinant in \eqref{vdm} can be changed into the well-known Vandermonde determinant:
 	\begin{equation*}\label{vdm2}
 \begin{small}
 \begin{vmatrix}
 1&1&\cdots&1\\
 2\cos \theta_1&2\cos \theta_2&\cdots&2\cos \theta_m\\
 (2\cos \theta_1)^2&(2\cos \theta_2)^2&\cdots&(2\cos \theta_m)^2\\
 \vdots&\vdots&&\vdots\\
 (2\cos \theta_1)^{m-1}&(2\cos \theta_2)^{m-1}&\cdots&(2\cos\theta_m)^{m-1}\\
 \end{vmatrix}.
 \end{small}
\end{equation*}
This proves Eq.~\eqref{vdm}.
\end{proof}
\noindent\emph{Proof  of Theorem \ref{up}}  Using the notations in Remark \ref{usm}, we have
\begin{equation*}
[\xi_1,\ldots,\xi_{\frac{n}{2}-1};\xi_{\frac{n}{2}+1},\ldots,\xi_{n-1}]=\left(\sum\limits_{i=0}^{\frac{n}{4}-1}U^{4i}\right)[\tau_1,\ldots,\tau_{\frac{n}{2}-1};\tau_{\frac{2}{2}+1},\ldots,\tau_{n-1}].
\end{equation*}
Noting that $\sum_{i=0}^{\frac{n}{4}-1}U^{4i}$ is a unit upper triangular matrix, we find that the two matrices $[\xi_1,\ldots,\xi_{\frac{n}{2}-1};\xi_{\frac{n}{2}+1},\ldots,\xi_{n-1}]$ and $[\tau_1,\ldots,\tau_{\frac{n}{2}-1};\tau_{\frac{2}{2}+1},\ldots,\tau_{n-1}]$ have the same determinant. Combining this fact and  Lemma \ref{detcos} leads to
 \begin{eqnarray*}
&&\det[\xi_1,\ldots,\xi_{\frac{n}{2}-1};\xi_{\frac{n}{2}+1},\ldots,\xi_{n-1}]\\
&=&\begin{small}
	\begin{vmatrix}
			2\cos (n-3)\alpha_1&\cdots&2\cos (n-3)\alpha_{\frac{n}{2}-1}&2\cos (n-3)\alpha_{\frac{n}{2}+1}&\cdots&2\cos (n-3)\alpha_{n-1}\\
			\vdots&\vdots&\vdots&\vdots&&\vdots\\
			2\cos 2\alpha_1&\cdots&2\cos 2	\alpha_{\frac{n}{2}-1}&2\cos 2\alpha_{\frac{n}{2}+1}&\cdots&2\cos 2\alpha_{n-1}\\
			2\cos \alpha_1&\cdots&2\cos \alpha_{\frac{n}{2}-1}&2\cos \alpha_{\frac{n}{2}+1}&\cdots&2\cos \alpha_{n-1}\\
			1&\cdots&1&1&\cdots&1\\
	\end{vmatrix}
\end{small}\\
&=&(-1)^\frac{(n-2)(n-3)}{2}\prod\limits_{\substack{1\le k_1< k_2\le n-1\\k_1,k_2\neq \frac{n}{2}}}(2\cos\alpha_{k_2}-2\cos\alpha_{k_1}).
 \end{eqnarray*}
Note that $2\cos\alpha_k$'s are all eigenvalues of $B^\T$ and $\xi_k$'s are the corresponding eigenvectors. It follows from Lemmas \ref{relWA} and \ref{etxi} that
\begin{equation*}
\det W(B)=\frac{\prod\limits_{\substack{1\le k_1<k_2\le n-1\\k_1,k_2\neq\frac{n}{2}}} (2\cos \alpha_{k_2}-2\cos\alpha_{k_1})	\prod\limits_{\substack{k=1\\k\neq\frac{n}{2}}}^{n-1}e^\T\xi_k}{\det[\xi_1,\ldots,\xi_{\frac{n}{2}-1};\xi_{\frac{n}{2}+1},\ldots,\xi_{n-1}]}=\pm 2^{\frac{n}{2}-1}.
\end{equation*}
Thus, by Corollary \ref{wdb}, we have $\det \tilde{W}(D_n)=\pm 2^{\frac{n}{2}-1}$ and Theorem  \ref{up} follows.
\end{proof}

\section*{Declaration of competing interest}
There is no conflict of interest.

	\section*{Acknowledgments}
 The author would like to thank Zhibin Du for his constructive criticism of the manuscript. This work is supported by the	National Natural Science Foundation of China (Grant Nos. 12001006 and 11971406) and the Scientific Research Foundation of Anhui Polytechnic University (Grant No.\,2019YQQ024).


\begin{thebibliography}{99}
	\bibitem{choi2021LAA} J.~Choi, S.~Moon, S.~Park, A note on the invariant factors of the walk matrix of a graph, Linear Algebra Appl.  631 (2021) 362-378.
	
	
	
\bibitem{chandler2017DCC}	D.B.~Chandler, P.~Sin, Q.~Xiang, 	The Smith group of the hypercube graph, Des. Codes Cryptogr. 84 (2017) 283-294.


\bibitem{godsil2010}	C.~Godsil,	Controllable subsets in graphs,	Ann. Comb. 16 (2012) 733-744.

\bibitem{hou2011ACS}Y.~Hou, W.C.~Shiu, W.H.~Chan,	Graphs whose critical groups have larger rank, Acta Mathematica Sinica 27 (2011) 1663-1670

\bibitem{lorenzini2008JCTB}D. Lorenzini, Smith normal form and Laplacians, J. Combin. Theory, Ser. B 98 (2008) 1271-1300.




	\bibitem{mao2015LAA}L. Mao, F. Liu, W. Wang, A new method for constructing graphs determined by their generalized spectrum, Linear Algebra Appl.
	477 (2015)112-127.
		\bibitem{pantangi2019JCTA} V.R.T.~Pantangi, P.~Sin, Smith and critical groups of polar graphs, J. Combin. Theory, Ser. A 167 (2019) 460-498.
 \bibitem{qiu2023DM} L.~Qiu, W.~Wang, W.~Wang, H.~Zhang, Smith Normal Form and the generalized spectral characterization of graphs,  Discrete Math. 346 (2023) 113177.
	\bibitem{rivlin} T.~J.~Rivlin, Chebyshev Polynomials: From Approximation Theory to Algebra and Number Theory, John Wiley, New York, 1990.

	



	\bibitem{wang2021LAA} W.~Wang, On the Smith normal form of walk matrices, Linear Algebra Appl. 612 (2021) 30-41.
	\bibitem{wang2022LAA} W.~Wang, C.~Wang, S.~Guo, On the walk matrix of the Dynkin graph $D_n$, Linear Algebra Appl. 653 (2022) 193-206.

	
	\bibitem{wang2006PHD} W.~Wang, On the spectral characterization of graphs, PhD Thesis, Xi'an Jiaotong University, 2006 (in Chinese).
	\bibitem{wang2014Ejc} W.~Wang, Generalized spectral characterization of graphs revisited, Electronic J. Combin. 20(2014),\#P4.
	\bibitem{wang2017JCTB} W. Wang, A simple arithmetic criterion for graphs being determined by their generalized spectra, J. Combin. Theory, Ser. B 122 (2017) 438-451.
	\bibitem{wang2021Eujc}	W.~Wang, F.~Liu, W.~Wang, Generalized spectral characterizations of almost controllable graphs, European J. Combin. 96(2021) 103348.
	\bibitem{wang2023Eujc} W.~Wang, W.~Wang, F.~Zhu, An improved condition for a graph to be determined by its generalized spectrum, European J. Combin. 108 (2023) 103638.
\end{thebibliography}
\end{document}